\documentclass{amsart}
\usepackage{tikz}

\newtheorem{definition}{Definition}
\newtheorem{theorem}[definition]{Theorem}
\newtheorem{proposition}[definition]{Proposition}
\newtheorem{lemma}[definition]{Lemma}
\newtheorem{claim}{Claim}
\newtheorem*{claim*}{Claim}
\newtheorem*{observation*}{Observation}
\title[Helly-type theorems for hypergraphs]{Helly-type theorems for the ordering\\ of the vertices of a hypergraph}
\author[C. Bir\'o]{Csaba Bir\'o}  
\address{Department of Mathematics University of Louisville, Louisville, KY 40292}
\email{csaba.biro@louisville.edu}
\author[J. Lehel] {Jen\H o Lehel}
\address{Department of Mathematics University of Louisville, Louisville, KY 40292}
\email{jeno.lehel@louisville.edu}
\address{Alfr\'ed R\'enyi Institute of Mathematics Budapest, Hungary}
\email{lehelj@renyi.hu}
\author[G. T\'oth] {G\'eza T\'oth}\thanks{}
\address{Alfr\'ed R\'enyi Institute of Mathematics Budapest, Hungary}
\email{toth.geza@renyi.hu}
\thanks{G\'eza T\'oth was supported by National Research, Development and Innovation Office, NKFIH, K-131529 and  ERC Advanced Grant ``GeoScape,'' No. 882971.}
\begin{document}
\begin{abstract} 
  Let $H$ be a complete $r$-uniform  hypergraph such that two vertices are marked in each edge as its `boundary' vertices.
  A linear ordering of the vertex set of $H$ is called an {\em agreeing linear order}, provided  all vertices of each edge
  of $H$ lie between its two boundary vertices. We prove the following Helly-type theorem: if there is an {agreeing linear order}
  on the vertex set of  every subhypergraph of $H$ with at most $2r-2$ vertices, then there is an agreeing linear order on the
  vertex set of $H$. We also show that the constant $2r-2$ cannot be reduced in the theorem.
  The case $r=3$ of the theorem has particular interest in the axiomatic theory of betweenness.
Similar results are obtained for further $r$-uniform hypergraphs ($r\geq 3$),
 where  one  or two vertices are marked in each edge, and 
 the linear orders need to satisfy various rules of agreement.
 In one of the cases we prove that no such Helly-type statement holds.
\end{abstract}
\maketitle

\section{Introduction}
\label{intro}

``Interest in ternary betweenness relations stems from their use in axiomatizations of geometry in the latter part of the nineteenth century\dots'' writes Fishburn. \cite[p.~59]{Fishbook}
More than a century ago 
several systems of axioms have been established  by 
 Huntington and  Kline \cite{Hunt,HuntK}
to characterize the usual betweenness of triples of points of a real line. ``No axiom of Huntington and Kline uses more than four points,'' observed Fishburn \cite{Fish}, thus expressing a basic Helly-type property of a complete finite system of betweenness: if there is an `agreeing linear order' for every subsytem on $4$ points, then there is an agreeing linear order for the whole system.
 
The investigation in the present paper is a natural extension  of our ongoing study of betweenness of convex bodies in the plane \cite{BLT} (for further details see Section \ref{between}). Here the early axiomatic context is replaced with general Helly-type questions that are presented and discussed in terms of hypergraphs.

The closest theme to our subject in hypergraph theory concerns  the concept of an ordered hypergraph, that is a set system together with a linear ordering of its elements. 
  These combinatorial structures emerge in various contexts, from the study of matrices with forbidden submatrices \cite{A},
  to the modeling of combinatorial geometry problems \cite{Furedi}. 
  Ordered hypergraphs are considered mostly in the theory of extremal hypergraphs (see \cite{Fox,Hubard}).
Meanwhile, here we are dealing with different combinatorial problems on complete ordered hypergraphs, 
which originated in the study of betweenness of convex bodies in the plane.

Let $H$ be a complete uniform hypergraph with no repeated edges, that is, a {\it clique}. If every edge $e\in E(H)$ has a set  of two
`marked' vertices, denoted $\partial{e}$, then $H$ is called here 
a {\it $2$-extreme marked clique}. A linear order $L$ on $V(H)$ is called an \emph{agreeing linear order} of $H$, 
if for every $e\in E(H)$ the set of the $L$-minimal and the $L$-maximal elements in $e$ is equal to $\partial{e}$. 
Similarly, for a set of vertices $U\subseteq V(H)$, a linear order $L$ on $U$ is called an \emph{agreeing linear order} of $U$, 
if for every $e\in E(H)$, $e\subseteq U$, 
the $L$-minimal and the $L$-maximal elements in $e$ is equal to $\partial{e}$. 
 
{A $2$-extreme marked clique with its vertex set ordered according to an agreeing linear order, if exists, becomes an ordered hypergraph.}

\begin{theorem}
\label{main}
Let  $H$ be a $2$-extreme marked $r$-uniform  clique on at least $2r-2$ vertices ($r\geq 3$). 
If every subset of $2r-2$ vertices of $V(H)$ has an agreeing linear order, then $H$ has an agreeing linear order. 
Furthermore,  there exists a clique $H$  such that every set of $2r-3$ vertices of $H$ has an agreeing linear order, but  $H$ does not.
\end{theorem}

In  Section~\ref{2extreme} we prove Theorem~\ref{main}, by using two different approaches. Section~\ref{proof3} proves the case $r=3$
separately (Theorem~\ref{main3}); we could not find a convenient extension of this `direct' proof for $r\geq 4$.  Section~\ref{proof4}
contains a proof for  $r\geq 4$ (Theorem~\ref{main4} through Lemmas~\ref{unique} and~\ref{extremals}); this approach does not work for
$r=3$ without making the proof of Lemma~\ref{extremals} much less transparent. 
Section~\ref{between} comments on Theorem~\ref{main3} (the case $r=3$) that has particular interest in the axiomatic theory of
betweenness (see \cite{AzNa,Fish,Hunt,HuntK}). 

Actually, a $3$-uniform $2$-extreme marked clique $H$ can be considered as `one-marked' 
  by marking the `middle' vertex $\widehat{e}=e\setminus\partial{e}$ for every $e\in E(H)$; 
  and an agreeing linear order requires that  $\widehat{e}$ be positioned between the two vertices of $\partial{e}$,
  for every $e\in E(H)$.
In Sections ~\ref{minmarked},~\ref{1extreme} and \ref{2marked}
we extend the concept of marked hypergraphs, where one or two vertices are marked in each edge,
and the marked vertices agree with particular rules in the linear ordering of the vertex set.

Given an $r$-uniform clique with edges containing one or two marked vertices, our main interest
consists in finding Helly-type theorems guaranteeing the existence of an agreeing linear order for the hypergraph.
 {We supply a table of content of the paper that may help the reader familiarize the non-standard notions}.
 \begin{table}[h!]
\centering
{\tiny\bf
 \begin{tabular}{c|c|c|c|r}
  clique& marked &  agreeing lin. ord.& Helly-number&  \\\hline\hline
   $2$-extreme marked&  $\partial{e}$&  $\partial{e}$ has the min. and max. of $e$  & $2r-2$& Section~\ref{2extreme} \\\hline 
     min-marked&  $A(e)$&$A(e)$ is the min. in $e$ & $r+1$&Section~\ref{minmarked} \\\hline 
      $1$-extreme marked&  $\widehat{e}$& $\widehat{e}$ is the min. or max. in $e$ & ---&Section~\ref{1extreme} \\\hline 
        min\&max-marked&  $\{A(e),B(e)\}$&$A(e)$ is min. $B(e)$ is max. in $e$ & $2r-2$ & Section~\ref{2marked}\\\hline 
 \end{tabular}
 \label{map}
 }
 \end{table}

In Section~\ref{minmarked} we consider  $r$-uniform min-marked  cliques,
with one vertex $A(e)$ marked in each edge $e$,  and in an agreeing linear order $A(e)$ is minimal among the vertices of $e$.
 We obtain  a Helly-type theorem whose proof reveals a characterization of min-marked  hypergraphs $H$ that have an agreeing linear order, in general, 
 in terms of a forbidden $2\times 2$ submatrix in the incidence matrix of $H$ (Theorem~\ref{char}).

\begin{theorem}
\label{leftentry}
Let  $H$ be a min-marked  $r$-uniform clique on at least $r+1$ vertices ($r\geq 3$).
Then the vertices of $H$ have an agreeing linear order 
if and only if each subhypergraph of $H$ with $r+1$ vertices has an agreeing linear order.
\end{theorem}

In Section~\ref{1extreme} the agreeing linear order  of a $1$-extreme marked hypergraph is investigated that requires the marked  vertex $\widehat{e}$ be either the minimal or the maximal among the vertices in each edge $e$. 
A straightforward characterization (Proposition~\ref{1exchar}) leads to the unexpected  fact that 
there is no Helly-type theorem  for the existence of an agreeing linear order for  $1$-extreme marked cliques (Proposition~\ref{general}). 

Another Helly-type result concludes the paper by characterizing those min\&max-marked cliques which admit an
 agreeing linear order where both the minimum and the maximum vertices are prescribed in each edge. 
\begin{theorem}
\label{leftrightentry}
Let $H$ be a min{\rm \&}max-marked $r$-uniform clique with 
at least $2r-2$ vertices ($r\geq 3$). Then $H$ has an agreeing linear order 
if and only if each subhypergraph of $H$ with $2r-2$ vertices has an agreeing linear order.
\end{theorem}

We prove Theorem ~\ref{leftrightentry} as a corollary of Theorem~\ref{main}, and we also give an independent graph theory proof, both in Section~\ref{2marked}.

Note that in the $2$-extreme marked and $1$-extreme marked cases, if $L$ is an agreeing linear order, then the dual of $L$ (in which every relation is reversed) is also agreeing. We will use the notation $L^d$ to denote the dual of the linear order $L$.

\section{$2$-extreme marked cliques}
\label{2extreme}

Let $H$ be an $r$-uniform  clique. If every edge of $H$ has two marked vertices, then $H$ is called here 
a {\it $2$-extreme marked clique}.  {The name `$2$-extreme marked' indicates that the marked vertices of an edge should become the maximum and the minimum among the vertices of each edge in an agreeing linear order.}
The set of the two marked vertices of an edge $e$  is denoted by $\partial e$. 

Let $U\subseteq V(H)$ and let
$L$ be a linear ordering on $U$. We say that  an edge $\{v_1,\ldots,v_r\}\subseteq U$
  \emph{agrees} with $L$, provided $v_1<_Lv_2<_L\cdots<_Lv_r$ and $\partial e=\{v_1,v_r\}$; furthermore, 
  $L$ is called an \emph{agreeing linear order} on $U$, 
 if every edge entirely in $U$ agrees with $L$. 
 
In the next sections we prove Theorem~\ref{main},  a Helly-type theorem on the existence of an agreeing linear order for $2$-extreme marked cliques. The special case $r=3$ is proved separately in Section~\ref{proof3}, partly because, as mentioned in Section~\ref{between},  this case has a  particular interest in the axiomatic theory of betweenness (see \cite{AzNa,Fish,Hunt,HuntK}). It is worth noting that
we could not find a convenient extension of the `direct' proof of this special case for $r\geq 4$. The proof of the general case in Section~\ref{proof4} uses a different approach.

\subsection{The case $r=3$}
\label{proof3}

Here we prove  the special case $r=3$ of Theorem~\ref{main}. 

\begin{theorem}
\label{main3}
Let $H$ be a $2$-extreme marked $3$-uniform clique. If every set of $4$ vertices of $H$ has an agreeing linear order, then $H$ has an agreeing linear order. 
\end{theorem}

\begin{proof}
  Let $x\in V(H)$. We define a binary relation $\sim\;=\;\sim_x$ on $V(H)\setminus\{x\}$ with respect to $x$: let $u\sim v$ if $u=v$ or $u\neq v$ and $\partial\{x,u,v\}\neq\{u,v\}$.
We show that the relation $\sim$ is an equivalence relation.

Only transitivity is nontrivial. Let $u\sim v$ and $v\sim w$. Suppose for a contradiction that $\partial\{x,u,w\}=\{u,w\}$. The set $\{x,u,v,w\}$ has an agreeing linear order $L$. In $L$, we have $x$ between $u$ and $w$; due to symmetry we may assume $u<_Lx<_Lw$. So where is $v$? Since $u\sim v$, we must have $v<_L x$, but since $v\sim w$, we must have $v>_L x$, a contradiction.

We claim  that 
there are at most two equivalence classes of the relation $\sim$. 
We will prove this by showing that 
$u\not\sim v$ and $v\not\sim w$ imply $u\sim w$. Suppose this is not true, and chose vertices with $u\not\sim v$, $v\not\sim w$, and $u\not\sim w$. The set $\{x,u,v,w\}$ has an agreeing linear order $L$. Similarly as above, we may assume $u<_L x<_L w$. This time, $u\not\sim v$ implies $v>_L x$, and $v\not\sim w$ implies $v<_L x$ in $L$, a contradiction. 
Note also that  if $x$ is not in the boundary of some edge, say $x\not\in\partial\{x,u,v\}$, then there are exactly two equivalence classes, since  $\partial\{x,u,v\}=\{u,v\}$ implies $u\not\sim v$.

The proof of the theorem proceeds by induction on $|V(H)|$. If $|V(H)|=2$, the statement is trivial. Let $|V(H)|\geq 3$, and chose a vertex $x$ that is not in the boundary of some edge. Let the two equivalence classes with respect to $x$ be $A_1$ and $A_2$. By the hypothesis, $A_i\cup\{x\}$ has an agreeing linear order $L_i$ for $i=1,2$.

Observe that $x$ is the greatest or the least element of $L_1$ and $L_2$. Indeed, if  
$u,v\in A_i$ was such that $u<x<v$ in $L_i$, then $\partial\{x,u,v\}=\{u,v\}$, so $u\not\sim v$ contradicting $u,v\in A_i$. 
After possibly taking duals, we may assume that $x$ is the greatest element of $L_1$ and the least element of $L_2$. Concatenate $L_1$ and $L_2$ by adding every relation $A_1<A_2$ to form the linear order $L$ on $V(H)$. 

We conclude the proof by showing that $L$ is an agreeing linear order. 
Let $e\in E(H)$. If $e\subseteq A_i\cup\{x\}$ for some $i$, then $e$ agrees with $L$. If $e=\{u,x,v\}$ with $u\in A_1$ and $v\in A_2$, then $u\not\sim v$ exactly means $\partial\{u,x,v\}=\{u,v\}$, so $e$ agrees with $L$. 

The remaining case (up to symmetry) is $e=\{u,v,w\}$, $u,v\in A_1$, $u<v$ in $L_1$, and $w\in A_2$. 
The set $\{u,v,w,x\}$ has an agreeing linear order $L'$. Since $u\not\sim w$, we may assume (up to duality) that $u<x<w$ in $L'$. Since $v\not\sim w$, we have $v<x$ in $L'$. Note that $u<v<x$ in $L_1$ implies $\partial\{u,v,x\}=\{u,x\}$, so in $L'$, we must have $u<v<x<w$. Since $L'$ is agreeing, this shows $\partial\{u,v,w\}=\{u,w\}$, and thus $e$ agrees with $L$.
\end{proof}

\subsection{The case $r\geq 4$}
\label{proof4}

Here we restate Theorem~\ref{main} and prove it for $r\geq 4$. The proof uses two lemmas.
\begin{lemma}\label{unique}
Let $r\geq 4$, and assume that $L$ is an agreeing linear order for the $2$-extreme marked $r$-uniform clique $H$ with $n$ vertices.  Then $L$ is unique (up to duality) if and only if $n\geq 2r-3$.
\end{lemma}
\begin{proof}
Let $(u_1,\ldots,u_n)$ be an agreeing linear order of the vertices of $H$. If $n\leq 2r-4$ and $r\geq 4$, then $u_{\lfloor n/2\rfloor}$ and  $u_{\lfloor n/2\rfloor+1}$ are  not boundary vertices  in any edge. Therefore, they can be swapped to obtain another agreeing linear order. Because the agreeing linear order is not unique, $n\geq 2r-3$ follows.

Let $n\geq 2r-3$. Let $L_1=(u_1,\ldots,u_n)$ and $L_2$ be agreeing linear orders of $H$.

\begin{claim*}\ 
\begin{itemize}
\item[--] If $u_1<_{L_2}u_n$ then for all $i<j$, we have $u_i<_{L_2}u_j$;
\item[--] If $u_1>_{L_2}u_n$ then for all $i<j$, we have $u_i>_{L_2}u_j$.
\end{itemize}
\end{claim*}

\begin{proof}

Assume first $u_1<_{L_2}u_n$. If $i=1$ and $j=n$, the statement is trivial. If $i=1$ and $j<n$, then let $e$ be an edge with $\{u_1,u_j,u_n\}\subseteq e$; since $\partial e=\{u_1,u_n\}$ (based on $L_1$), we have that $u_j$ must be between $u_1$ and $u_n$ in $L_2$, hence $u_i=u_1<_{L_2}u_j<_{L_2}u_n$. Similar simple argument applies if $i>1$ and $j=n$. So the only case remains when $i>1$ and $j<n$.

Observe that because $(j-1)+(n-i)\geq n\geq 2r-3$, we have
\begin{equation}\label{eq:onesidelarge}
j\geq r\qquad\text{or}\qquad i\leq n-r+1.
\end{equation}

Suppose for a contradiction that $u_j<_{L_2}u_i$.
Let $A\subseteq V(H)\setminus \{u_1,u_i,u_j,u_n\}$ be an $(r-4)$-element subset of $V(H)$, and let $e=A\cup\{u_1,u_i,u_j,u_n\}$. Since $\partial e=\{u_1,u_n\}$, we have 
that $u_i$ and $u_j$ are both between $u_1$ and $u_n$ in $L_2$, so specifically,
\begin{equation}\label{eq:four}
u_1<_{L_2}u_j<_{L_2}u_i<_{L_2}u_n.
\end{equation}

Recalling (\ref{eq:onesidelarge}), suppose first that $j\geq r$. This means that there exists an edge $f\subseteq\{u_1,\ldots,u_j\}$ such that $\{u_1,u_i,u_j\}\subseteq f$. Since $\partial f=\{u_1,u_j\}$, it follows that $u_i$ is between $u_1$ and $u_j$ in $L_2$, contradicting (\ref{eq:four}). Now suppose that $i\leq n-r+1$. Then there exists $g\subseteq\{u_i,\ldots,u_n\}$ such that $\{u_i,u_j,u_n\}\subseteq g$. Since $\partial g=\{u_i,u_n\}$, it follows that $u_j$ is between $u_i$ and $u_n$ in $L_2$, contradicting (\ref{eq:four}).

Following a similar argument for $u_1>_{L_2}u_n$, equation (\ref{eq:four}) turns into
\[
u_n<_{L_2}u_i<_{L_2}u_j<_{L_2}u_1,
\]
and a similar argument to the one above leads to a contradiction.
\end{proof}

It is now easy to see how this technical claim implies the Lemma. If $u_1<_{L_2}u_n$, then every pair of vertices is ordered the same in $L_2$ as in $L_1$, so $L_1=L_2$. If $u_1>_{L_2}u_n$, then every pair of vertices is ordered the opposite in $L_2$ as in $L_1$, so $L_2=L_1^d$.
\end{proof} 

A vertex $\ell\in V(H)$ is an \emph{extremal} vertex of a $2$-extreme marked clique $H$, if every $e\in E(H)$  containing $\ell$ satisfies $\ell\in\partial e$. Note that $H$ has at most two extremal vertices. 

\begin{lemma}
\label{extremals}
For $r\geq 4$, let $H$ be a $2$-extreme marked $r$-uniform  clique with  $|V(H)|\geq 2r-2$.
If every set of $2r-2$ vertices of $H$ has an agreeing linear order, then $H$ has exactly two extremal vertices.
\end{lemma}

\begin{proof}
  We proceed by induction on the number $n$ of vertices. If $n=2r-2$, then $H$ has an agreeing linear order $L$.
  The least element $\ell_1$ and the greatest element $\ell_2$ of $L$ are clearly  extremal vertices of $H$.

Now let $n>2r-2$, and suppose that the lemma is true for $n-1$. Let $x\in V(H)$ be a non-extremal vertex of $H$,
and let $H'=H-x$. We apply the hypothesis to find extremal vertices $\ell_1$ and $\ell_2$ in $H^\prime$.
Next consider the edge set
\[
F=\{\{x,\ell_1,\ell_2,v_1,\ldots,v_{r-3}\}:v_1,\ldots,v_{r-3}\in V(H)\setminus\{x,\ell_1,\ell_2\}\}.
\]

\begin{claim}
\label{claim1}
For some  $f\in F$ we have $x\not\in\partial f$.
\end{claim}

\begin{proof}
  Since $x$ is not extremal in $H$, we have $x\not\in\partial e$ for some $e\in E(H)$ containing $x$. 
  Let $S=e\cup\{\ell_1, \ell_2\}$. If $S=e$ then we are done, because in this case $e\in F$. Hence we have
 $r+1\leq |S|\leq r+2\leq 2r-2$. 
 
 Let $L$ be an agreeing  linear order on $S$. 
Because $x\not\in\partial e$, the vertex  $x$ is not the minimal or maximal
element of $S$ in $L$.  
If $\ell_1$ and $\ell_2$ are  the minimal and maximal elements of $S$ in $L$, then there is a subset $f\subseteq S$, $|f|=r$, containing $\{x,\ell_1,\ell_2\}$, that is,  $f$ is an edge in $F$
satisfying  $x\not\in\partial f$, as required. We show that this is the case.

Suppose on the contrary that  $a\in e\setminus\{x,\ell_1, \ell_2\}$ is the $L$-minimal or $L$-maximal element of $S$.
  Take any subset $g\subseteq S$,
  $|g|=r$, such that $a\in g$ and $x\not\in g$. Then $\partial g=\{\ell_1,\ell_2\}$, because $g$ is an edge in $H^\prime$.    
   Since $L$ is an agreeing linear order on $g\subseteq S$,  
and  $a\in g$ is a minimal or maximal element of $S$ in $L$, we obtain $a\in\partial g$, 
 a contradiction.
\end{proof}

\begin{claim}
\label{claim2}
For every  $f\in F$ we have 
$x\not\in\partial f$.
\end{claim}

\begin{proof}
By Claim~\ref{claim1}, $x\not\in\partial f$ 
 for some $f\in F$.
Assume for a contradiction that  $x\in\partial f^\prime$  for some $f^\prime\in F$.  
Let these edges be
\[
f=\{x,\ell_1,\ell_2,v_1,\ldots,v_{r-3}\},\qquad f'=\{x,\ell_1,\ell_2,v_1',\ldots,v_{r-3}'\}.
\]
Notice that this step assumes $r\geq 4$. 
Since $|f\cup f'|\leq 2r-3$, we have an agreeing linear order $L$ on $f\cup f'$.

Let $W$ be any $(r-2)$-element subset of $V=\{v_1,\ldots,v_{r-3},v_1',\ldots,v_{r-3}'\}$.
Then, by the hypothesis, $\partial(W\cup\{\ell_1,\ell_2\})=\{\ell_1,\ell_2\}$. This means that 
in $L$, every element of $V$ is between $\ell_1$, and $\ell_2$. On the other hand, due to $x\in\partial f^\prime$,
we have that $x$ is not between $\ell_1$ and $\ell_2$. So up to symmetry and the possible exchange of $\ell_1$ and $\ell_2$, 
 we obtain 
\[
\ell_1<_LV<_L\ell_2<_Lx
\]
Then $f\subseteq V\cup\{x,\ell_1,\ell_2\}$ implies $x\in\partial f$, a contradiction.
\end{proof}

We show that both $\ell_1$ and $\ell_2$ are extremal in $V(H)$. 
It is enough to show the statement for $\ell_1$; the proof for $\ell_2$ is similar. If $\ell_1$ is not extremal, there exists an edge $f=\{\ell_1,x,v_1,\ldots,v_{r-2}\}$ for which $\ell_1\not\in\partial f$. The set $\{x,\ell_1,\ell_2,v_1,\ldots,v_{r-2}\}$ has an agreeing linear order $L$.

Let $V=\{v_1,\ldots,v_{r-2}\}$, and let $W$ be any $(r-3)$-element subset of $V$. Since $\partial (V\cup\{\ell_1,\ell_2\})=\{\ell_1,\ell_2\}$, we have that every element of $V$ is between $\ell_1$ and $\ell_2$ in $L$. From the fact that $\ell_1\not\in\partial f$, some $v_i$ and $x$ must surround $\ell_1$ in $L$. So up to symmetry, 
\[
x<_L\ell_1< _LV<_L\ell_2
\]
Since $e=W\cup\{x,\ell_1,\ell_2\}\in F$, this contradicts $x\not\in\partial e$ proved in Claim~\ref{claim2}.
\end{proof}

\begin{observation*}
If $\ell\in V$ is an extremal vertex of $V$, then $\ell$ is the minimal or maximal vertex in every agreeing linear order on each $U\subseteq V$ containing $\ell$.
\end{observation*}

\begin{theorem}
\label{main4}
For $r\geq 4$, let $H$ be a $2$-extreme marked $r$-uniform clique with at least $2r-2$ vertices. Then $H$ has an agreeing linear order if and only if every subhypergraph of $H$ on $2r-2$ vertices has an agreeing linear order. Furthermore, the constant $2r-2$  cannot be replaced by a smaller value. 
\end{theorem}

\begin{proof}  Necessity  in the first claim is obvious, we prove the sufficiency by induction on the order of $H=(V,E)$.
  The claim is true for $|V|=2r-2$, by the condition. Let $|V|=n+1$, $n\geq 2r-2$,
  and assume that the claim is true for $n$ vertices. We may assume, by Lemma~\ref{extremals},  that $v_0, v_n$ are the extremal vertices of $H$. By induction, there is  an agreeing linear order $L_1$ for $H-v_0$. Due to the Observation, up to duality, we may assume that $L_1=(v_1, \ldots,v_n)$.
We claim that $L_0=(v_0,v_1, \ldots,v_n)$ is an agreeing linear order for $H$.

Let $e$ be an arbitrary edge of $H$. If $v_0\not\in e$, then $e$ agrees with $L_1$ so $e$ agrees with $L_0$. For the balance of the argument, assume $v_0\in e$, specifically, let
$e=\{v_0,v_{i_1},v_{i_2}, \ldots, v_{i_{r-1}}\}$ with $1\leq{i_1}< \ldots <i_{r-1}\leq n$.
Consider a $(2r-2)$-element set $U$ containing $e$ and including $v_n$. 
By the condition of the theorem,
$U$ has an agreeing linear order $L_0'$. By the Observation, $L_0'$ has $v_0$ and $v_n$ as its minimal and maximal elements. Let $U'=U\setminus\{v_0\}$. Since $|U'|=2r-3$, by Lemma~\ref{unique}, $L_0'|_{U'}=L_1|_{U'}$ or $L_0'|_{U'}=(L_1|_{U'})^d$. Hence, up to duality, in $L_0'$ we have
\[
v_0<v_{i_1}<\cdots<v_{i_{r-1}}<v_n,
\]
therefore $\partial  e=\{v_0,v_{i_{r-1}}\}$.
 
To see the second claim of the theorem we present a hypergraph $H$ on  $n=2r-2$ vertices, for every $r\geq 4$, such that
the vertex set of each subhypergraph of $H$ with $2r-3$ has an agreeing linear order but $V(H)$ does not.
Let 
 $V(H)=\{v_1,v_2,\ldots,v_{2r-2}\}$; for $e_0=\{v_1,\ldots,v_{r-1},v_r\}$ define $\partial  e_0=\{v_{1},v_{r-1}\}$, and for every   $e=\{v_{i_1},v_{i_2},\dots,v_{i_r}\}$, $1\leq i_1<i_2<\ldots i_r\leq 2r-2$,  different from $e_0$ define $\partial  e=\{v_{i_1},v_{i_r}\}$. 
 Let  $e_1=\{v_{r-1},v_r,\ldots,v_{2r-2}\}$. 
Because the extremal vertices of $H$ are $v_1$ and $v_{2r-2}$, it follows that no agreeing linear order $L=(v_1,\ldots,v_{2r-2})$ agrees with both $e_0$ and $e_1$.  

Observe that $L_1=(v_1,\ldots,v_{r-1},v_r,\ldots,v_{2r-2})$ agrees with all edges of $H$ but $e_0$, and 
$L_0=(v_1,\ldots,v_{r},v_{r-1},\ldots,v_{2r-2})$  agrees with all edges of $H$ but $e_1$.
 Because  $e_0\cup e_1=V(H)$, no subhypergraph $H-v_i$ contains both $e_0$ and $e_1$; therefore, either $L_0-v_i$ or  $L_1-v_i$ is an agreeing linear order of $V(H)\setminus\{v_i\}$, for every $1\leq i\leq 2r-2$. 
 \end{proof}
 
 \subsection{Betweenness structure}
 \label{between}
 
 Theorem~\ref{main3} on $2$-extreme marked $3$-uniform cliques has particular interest in the axiomatic theory of betweenness (see \cite{AzNa,Fish,Hunt,HuntK}). The edges of a  $2$-extreme marked  $3$-uniform complete hypergraph describe a `betweenness structure' 
investigated first by Huntington and Klein \cite{HuntK}. They write:
{\it The `universe of discourse' of the present paper is the class of all well-defined
systems $(K, R)$ where $K$ is any class of elements $A, B, C, \ldots$
and $R$ is any triadic relation. The notation $R[ABC]$, or simply $ABC$,
indicates that three given elements $A, B, C$, in the order stated, satisfy the
relation $R$.}

The four basic postulates of `betweenness' due to Huntington and Klein \cite{HuntK} can be interpreted in a natural way in terms of $2$-extreme marked $3$-uniform cliques. Let $H=(K,E)$ be a $3$-uniform complete hypergraph with vertex set $K$, and edge set $E$, where each edge
has two  boundary vertices and the third vertex is its `middle' vertex between the boundary vertices. Every $e=\{A,B,C\}\in E$ with two boundary vertices can be specified  as a sequence $(A,B,C)$ or $(C,B,A)$, where $B$ is the middle vertex of $e$.

\begin{itemize}
\item{Postulate A:} 
if $ABC$ is true, then $CBA$ is true 

-- for  $\{A,B,C\}\in E$ any of the  notations $(A,B,C)$ or $(C,B,A)$ reads as {\bf `$B$ is between $A$ and $C$'}.

\item{Postulate B:} if $A, B, C$ are distinct, then at least one of the six possible
permutations will form a true triad 

-- $H=(K,E)$ is a complete hypergraph.

\item{Postulate C:} we cannot have $AXY$ and $AYX$ both true at the same time 

-- $H=(K,E)$ is a simple hypergraph and each edge has two fixed boundary vertices, that is $2$-extreme marked.

\item{Postulate D:} if $ABC$ is true, then the elements $A, B$, and $C$ are distinct 

--  $H=(K,E)$ is a $3$-uniform hypergraph.
\end{itemize}

The less general model of betweenness is the betweenness relation naturally defined by a linear order.
By adding further axioms to Postulates A, B, C,  D, Huntington \cite{Hunt} and Huntington and Kline \cite{HuntK} identify
several sets of independent axioms to characterize  betweenness  structures admitting `agreeing linear order'.
They observe (see also Fishburn \cite{Fish}) that no axioms 
concerning more than four elements need be assumed as fundamental.

This observation takes a precise form here due to Theorem~\ref{main3} (proved in Section \ref{proof3}):
a betweenness structure admits an agreeing linear order if and only if its every substructure on $4$ elements admits an agreeing linear order. In the light of this theorem all systems of axioms listed by Huntington \cite{Hunt} are various logical constructions that exclude the obstructing  $4$-point betweenness structures. 

In the study of betweenness structures involving triples of convex sets in the plane \cite{BLT} we introduced a different kind of `agreeing' linear order, where the `middle element' of each triple precedes the two `boundary elements'. This concept of betweenness  is extended in the next sections from triples to $r$-tuples ($r\geq 3$), 
 where  one  or two elements are marked in each; and 
 our investigations concern linear orders of the elements agreeing with various rules of agreement.
 
\section{Min-marked hypergraphs}
\label{minmarked}

{Let $H$ be an $r$-uniform hypergraph (not necessarily a clique), and let $A(e)\in e$ be the vertex marked in each $e\in E(H)$. 
This hypergraph will be called a {\it min-marked hypergraph} 
 indicating that $A(e)$ should become the minimum among the vertices of each $e$ in an agreeing linear order.} 
A linear order $L$  of $H$ is called 
an {\em agreeing linear order}, provided  
$A(e)<_L v$, for every $e\in E(H)$ and $v\in e\setminus\{A(e)\}$.

First we will characterize general hypergraphs admitting an agreeing linear order, then we discuss the Helly-property of complete hypergraphs.

Let $M(H)$ denote the incidence matrix of $H$, that is, rows correspond to the edges, columns correspond to the vertices,
and for any $e\in E(H)$ and $\alpha\in V(H)$

$$m(e,\alpha)=\left\{
\begin{array}{rl}
0 &\hbox{if $\alpha\notin e$}\\
-1 &\hbox{if $\alpha=A(e)$}\\
1 &\hbox{if $\alpha\in e\setminus \{A(e)\}$}
\end{array}
\right.
$$

The matrix $F=
\begin{tabular}{|r|r|}\hline
   -1 & 1   \\\hline
   1 & -1\\\hline
\end{tabular}$ and its permutation will be called a {\it forbidden $2\times 2$}. 
If  $L$
 is an agreeing  linear order of the min-marked hypergraph $H$, then the incidence matrix $M(H)$ contains no submatrix equivalent to a forbidden  $2\times 2$.  
 
 The submatrix 
  $$
\begin{tabular}{c||r|r|r}
&$\alpha$&$\beta$\\\hline\hline
 $e$&  -1 & 1   \\\hline
 $f$ &  0 & -1\\\hline
\end{tabular}$$ and its permutations are called {\it precedence $2\times 2$} matrices.
If   $M(H)$ has a precedence submatrix as above and $L$
 is an agreeing  linear order, then the first row implies the `precedence' 
 $\alpha<_L \beta$; 
 
 We show 
 a necessary and sufficient condition for the existence of an agreeing
 linear order of a min-marked hypergraph. For this purpose we introduce an auxiliary directed graph
 $D(H)$, called the {\it $D$-graph of $H$}, containing `precedence information' about $H$ as follows. 
 Consider the partition $\Pi$ of $E(H)$ into classes where $f_1, f_2$ are equivalent if and only
 if $A(f_1)=A(f_2)$.
 Let  the vertices of $D(H)$ be the equivalence classes of $\Pi$, and for distinct $F_1,F_2\in\Pi$ let $F_1\rightarrow F_2$
 be a directed edge  of $D(H)$ 
 if $e\in F_1,f\in F_2$ such that $A(f)\in e$, {that is the $\{e,f\}\times\{A(e),A(f)\}$ submatrix of $M(H)$ is either a forbidden or a precedence matrix}.
 
\begin{proposition}
\label{digraph}
A min-marked hypergraph $H$ has an agreeing linear order if and only if the $D$-graph of $H$ contains no directed cycle.
\end{proposition}
\begin{proof} The condition is obviously necessary. Assume now that $D=D(H)$ has no directed cycle. Then 
  the vertices of $D$ have a topological order $L$,
  that is, an  ordering such that all edges go `forward' (from smaller to larger).
  Use this ordering of the equivalence classes of $\Pi$ to define a linear order of $V(H)$ by listing the marked vertices common to the edges in each  equivalence class according to the forward ordering, then all vertices of $H$ marked for no edge follow in arbitrary order.
  
  Suppose that it is not an agreeing linear order. Then there is an $e\in E(H)$, $\alpha, \beta\in e$, $\beta<_L\alpha$ and $\alpha=A(e)$.
  By the definition of $L$, there is an $f\in E(H)$ with $\beta=A(f)$. Suppose that $e\in F_1$, $f\in F_2$. 
  Then $A(f)\in e$, therefore, there is a directed edge from $F_1$ to $F_2$, contradicting $\beta<_L\alpha$. Therefore, $L$ is 
  an agreeing linear order of $H$.
\end{proof}

 \subsection{Min-marked cliques}
\begin{lemma} 
\label{r+1}
Assume that  $H$ is a min-marked $r$-uniform  clique with at least $r+1$ vertices  ($r\geq 3$). 
If $M(H)$ has a forbidden $2\times 2$, then there exist $e,f\in E(H)$ and $\alpha,\beta\in V(H)$ such that
$|e\cup f|=r+1$ and $\{e,f\}\times \{\alpha,\beta\}$ is a forbidden $2\times 2$.
\end{lemma}
\begin{proof}
Let $e,f\in E(H)$ be distinct edges and $\alpha,\beta\in V(H)$ such that
$\{e,f\}\times \{\alpha,\beta\}$ is a forbidden $2\times 2$; furthermore, 
let $k=|e\cup f|$ is as small as possible. We show that $k<r+2$ as stated.
Assume that $k>r+1$.

\begin{table}[h!]
\centering
\begin{tabular}{c||c|c|c|c|c|r|r|l|c|c|c|c|c|c|}
 H & &&&&$\alpha$&$\beta$&&\ldots&&$\zeta$&$\beta^\prime$&&& \\\hline\hline
$e$& 1&\ldots&1&1&  -1 & 1  &  0  &\ldots&0&0&1&\ldots&1&1\\\hline
$f$ & 1&\ldots&1&1&  1 & -1&1&\ldots&1&1&0&\ldots&0&0\\\hline
 $f^\prime$   & 1&\ldots&1&1& 1&$\star$&1&\ldots&1&0&$\star$&\ldots&0&0 \\
\end{tabular}
\vskip1em
\caption{}
\label{long}
\end{table}
Table~\ref{long} shows $e$ and $f$ by listing first the elements in $e\cap f$ ending with $\alpha$ and $\beta$,
then the elements in $f\setminus e$ followed by the elements of $e\setminus f$.
Let $f^\prime=(f\setminus\{\zeta\})\cup\{\beta^\prime\}$, where $\zeta\in f\setminus e$ and $\beta^\prime\in e\setminus f$ are arbitrary vertices. Observe that $|f\cup f^\prime|=r+1<k$ implies that $f\cup f^\prime$ contains no forbidden $2\times 2$; in particular, $A(f^\prime)\in\{\beta,\beta^\prime\}$ (see Table ~\ref{long}). In each case we obtain a forbidden $2\times 2$ in $e\cup f^\prime$, which contradicts the minimality of $k$, 
because $|e\cup f^\prime|=|(e\cup f)\setminus\{\zeta\}|=k-1$. 
\end{proof}

\begin{lemma} 
\label{forbidden}
A min-marked $r$-uniform clique $H$ with $r+1$ vertices  ($r\geq 3$) has an agreeing linear order if and only if 
$M(H)$ contains no forbidden $2\times 2$.
\end{lemma}
\begin{proof} 
If $H$ has an agreeing linear order then $M(H)$ contains no forbidden $2\times 2$. Assume now that $M(H)$ has no forbidden $2\times 2$.\\

\begin{claim*}
The $-1$ entries of $M(H)$ are in two columns.
\end{claim*}
\begin{proof}{
By permuting rows and columns, we may assume that $M(H)$ has the following form:
\begin{itemize}
\item[--] the first row is $e_1=\begin{bmatrix} 0 & -1 & 1 & \ldots & 1\end{bmatrix}$;
\item[--] the second row is $e_2=\begin{bmatrix} -1 & 0 & 1 & \ldots & 1\end{bmatrix}$ or $e_2=\begin{bmatrix} 1 & 0 & -1 & \ldots & 1\end{bmatrix}$;
\item[--] each subsequent row $e_i$ has its $0$ entry at the $i$th position. 
\end{itemize}
}

\begin{table}[h!]
\centering
\begin{tabular}{c||r|r|r|r|r|r}
H  & $\alpha$&$\beta$&$\gamma$&\ldots&$\zeta$ &\ldots \\\hline\hline
$e_1$ &  0 & -1  &  1  &\ldots&1\\\hline
$e_2$ & -1 & 0& 1&\ldots&\\\hline
\\\hline
 $f$   & 1&1&&\ldots&-1&
\end{tabular}\hskip1.5cm
\begin{tabular}{c||c|r|r|l|c|c|c}
 H & $\alpha$&$\beta$&$\gamma$&$\delta$&\ldots&$\zeta$& \ldots \\\hline\hline
$e_1$ &  0 & -1  &  1  &1&\ldots&1\\\hline
$e_2$ &  1 & 0&-1&1&\ldots&1\\\hline
 $e_3$   & (-1)&1&0&1&\ldots&(-1)\\\hline
 $e_4$ &  $\star$ & $\star$&$\star$&0&\ldots&1
\end{tabular}
\end{table}
To prove the claim, assume to the contrary that the $-1$ entries are spread in three (or more) columns.  In the table on the left for the first case, let  $\zeta$ be the column different from the first two columns, $\alpha$ and $\beta$, containing a $-1$, say $m(f,\zeta)=-1$. Then $\{e_1,f\}\times \{\beta,\zeta\}$ is a forbidden $2\times 2$. 

For the second case let $e_3$ be the third row in the standard right table. Then either  $m(e_3,\zeta)=-1$, for some $\zeta$ different from $\alpha,\beta$ and $\gamma$, and  we are done as before ($e_3$ taking the role of $f$); or $m(e_3,\alpha)=-1$, as  $\alpha$ being the third column containing $-1$.

In this latter case consider the next row $e_4$ and column $\delta$ of the standard $M(H)$ (they exist, since $r+1\geq 4$, furthermore, $m(e_4,\delta)=0$).
If $m(e_4,\alpha)=-1$ then $\{e_2,e_4\}\times \{\alpha,\gamma\}$ is a forbidden $2\times 2$; if $m(e_4,\beta)=-1$ or $m(e_4,\gamma)=-1$ then $\{e_3,e_4\}\times \{\alpha,\beta\}$ or $\{e_1,e_4\}\times \{\beta,\gamma\}$ yields a forbidden $2\times 2$, respectively. 
\end{proof}

Due to the Claim we may assume that all $-1$ entries of $M(H)$ are in columns $\alpha$ and $\beta$. 
If $(\alpha,\beta,\ldots)$ is not an agreeing linear order of $H$, then there is an edge $e\in E(H)$ such that $m(e,\alpha)=1, m(e,\beta)=-1$.
Then $m(f,\beta)\neq 1$ for every $f\in E(H)$, since $m(f,\beta)= 1$ implies that $f\neq e$ and $m(f,\alpha)=-1$, thus  $\{e,f\}\times \{\alpha,\beta\}$ would be a forbidden $2\times 2$. Thus if $(\alpha,\beta,\ldots)$ is not an agreeing linear order of $H$, then  $(\beta,\alpha,\ldots)$ is so.
\end{proof}
We conclude the section with the proof of Theorem~\ref{leftentry} stated in a more technical form.
\begin{theorem}
\label{char}
For a min-marked $r$-uniform  clique $H$ with at least $r+1$ vertices ($r\geq 3$) the following statements are equivalent
\begin{itemize}
\item[(i)] $M(H)$ contains no forbidden $2\times 2$,
\item[(ii)] There is an agreeing linear order on every $(r+1)$-element subset of $V(H)$.
\item[(iii)] $H$  has an agreeing linear order,

\end{itemize}
\end{theorem}
\begin{proof}
Statement  (iii) obviously implies (i) and (ii), furthermore, (i) and (ii) are equivalent due to Lemma~\ref{forbidden}.
{To prove  the theorem we assume that (i) and (ii) are true (we know that they are equivalent), that is,  $M(H)$ contains no forbidden $2\times 2$ and   there is an agreeing linear order  on every $(r+1)$-element subset of $V(H)$. 
We need to show that $H$ has an agreeing linear order.}
 
This claim is verified by induction on $n$ proving that $M(H)$ has a  column 
that contains no entry equal to $1$, that corresponds to a minimal vertex. This is  true for $n=r+1$, since then $H$ has  an agreeing linear order.  Assume $n>r+1$, let $H^\prime=H-\chi$, where $\chi\in V(H)$ is arbitrary,  furthermore, assume that there is a column $\lambda$ in $M(H^\prime)$ not containing $1$.\\

\begin{claim}
$A(f)\in\{\chi,\lambda\}$, for every $f\in E(H)$ with $\chi,\lambda\in f$.
\end{claim} 
\begin{proof} 
Assume to the contrary that $m(f,\chi)=m(f,\lambda)=1$ for some $f\in E(H)$. Let $\zeta\in f\setminus\{\chi,\lambda\}$ such that $A(f)=\zeta$. 
Consider the columns labeled by the vertices in $f\cup\{\xi\}$, where $\xi\notin f$ is arbitrary. Let 
$g=(f\cup\{\xi\})\setminus \{\chi\}$  (see the table). 
\begin{table}[h!]
\centering
\begin{tabular}{c||r|r|r|r|r|r|c}
 H & $\chi$&$\lambda$&\ldots&$\zeta$&&$\xi$ & \\\hline\hline
$f$ &  1 & 1   &\ldots&-1&\ldots&0\\\hline
$g$ & 0 & -1& \ldots&1&\ldots&1\\\hline
\end{tabular}
\end{table}
Notice that $A(g)=\lambda $, because $g\in E(H^\prime)$ and  $\lambda$ is a  in $M(H^\prime)$ containing no $1$. 

For  $\zeta=A(f)$ the submatrix $\{f,g\}\times \{\lambda,\zeta\}$ is a forbidden $2\times 2$, a contradiction. Therefore, either  $m(f,\chi)=-1$ or  $m(f,\lambda)=-1$, and the claim follows. 
\end{proof}
Next we show by contradiction that either $\lambda$ or $\chi$ is a column not containing $1$ in $M(H)$. 

If $\lambda$ is not such a column of $M(H)$, then there is an edge  $\ell\in E(H)\setminus E(H^\prime)$ 
satisfying $m(\ell,\lambda)=1$ and $m(\ell,\chi)\neq 0$. By the Claim, $m(\ell,\chi)=-1$.

If $\chi$ is not such a column of $M(H)$ then there is an edge $h\in E(H^\prime)$ with $m(h,\chi)=1$. Observe that 
 $m(h,\lambda)\neq -1$, because otherwise, the submatrix $\{h,\ell\}\times \{\chi,\lambda\}$ is a forbidden $2\times 2$. Thus the Claim 
 implies $m(h,\lambda)=0$.
Let $\zeta=A(h)$, then $m(\ell,\zeta)=0$ since otherwise, $\{h,\ell\}\times \{\chi,\zeta\}$ is a forbidden $2\times 2$ (see the table).

\begin{table}[h!]
\centering
\begin{tabular}{c||r|r|r|r|r|}
K  & $\chi$&$\lambda$&$\zeta$&  \\\hline\hline
$\ell$ &  -1 & 1   &0&\ldots\\\hline
  $h$ &  1 & 0   &-1&\ldots\\\hline
$g$ & $\star$ & $\star$&1&\ldots\\\hline
&&&&
\end{tabular}
\end{table}

Let $g\in E(H)$ be an edge with $\chi,\lambda,\zeta\in g$. The Claim  implies that the marked vertex in $g$ is one of the $\star$ entries, in each case producing a forbidden $2\times 2$, 
{a contradiction. Therefore, either $\chi$ or $\lambda$ is a column of $M(H)$ not containing $1$.}

An agreeing linear order of $H$ is obtained by removing the vertex corresponding to the least 
column $\lambda^*\in \{\lambda,\zeta\}$, 
then by taking an agreeing linear order for the $(n-1)$-element set $V(H)\setminus\{\lambda^*\}$, finally by completing this order with $\lambda^*$ as the least element. 
\end{proof}

\section{$1$-extreme marked hypergraphs}
\label{1extreme}
Let $H$ be an $r$-uniform hypergraph (not necessarily a clique), and let $\widehat{e}\in e$ be a dedicated vertex in each edge $e\in E(H)$. 
 This hypergraph will be called a {\it $1$-extreme marked hypergraph}. {The name `$1$-extreme marked' indicates that the marked vertex of an edge should become either minimum or maximum among the vertices of each edge in an agreeing linear order.} 
 A linear order $L$ of a set $U\subseteq V(H)$ is  an {\it agreeing linear order} on $U$ provided $\widehat{e}$ is either $L$-minimal  or $L$-maximal among the vertices of $e$ for
every edge $e\subseteq U$. An agreeing linear order on $V(H)$ is also called  an {agreeing linear order of $H$}.

The edge/vertex incidence matrix $M(H)$  is defined for every $e\in E(H)$ and $\alpha\in V(H)$ by the entries
$$m(e,\alpha)=\left\{
\begin{array}{rl}
0 &\hbox{if $\alpha\notin e$}\\
-1 &\hbox{if $\alpha=\widehat{e}$}\\
1 &\hbox{if $\alpha\in e\setminus \{\widehat{e}\}$}
\end{array}
\right.
$$ 

\subsection{Auxiliary graphs}

Let $H$ be a $1$-extreme marked hypergraph, and let $L$ be an agreeing linear order of $H$.
{}
A $2\times 2$ submatrix of $M(H)$ equal to
$
\begin{tabular}{|r|r|}\hline
   -1 & 1   \\\hline
   1 & -1\\\hline
\end{tabular}
$  
or its permutation is called an $F$-matrix; a 
$2\times 2$ submatrix of $M(H)$ equal to
$
\begin{tabular}{|r|r|}\hline
   1 & -1   \\\hline
   1 & -1\\\hline
\end{tabular}
$ or its permutation 
is called an $S$-matrix.
The two vertices corresponding to 
the columns  of an $F$-matrix
are  extremes of different types, one is the $L$-minimal, the other is the $L$-maximal element of the edges corresponding to the rows. 
Meanwhile,  
the column containing $-1$ of an $S$-matrix is either $L$-minimal or $L$-maximal element in both edges.

We associate a graph $SF(H)$  to $H$ in two steps as follows. First let
 $G$ be a graph with $V(G)=E(H)$ and for $e,f\in E(H)$ let $ef$ be an edge in $G$  labeled with $S$ or $F$ 
 if and only if there exists an $S$-matrix or an $F$-matrix 
with rows $e$ and $f$, respectively.  
Now $SF(H)$ is obtained from $G$ by contracting all $S$-edges, then  eliminating multiple $F$-edges and $S$-loops.
\begin{table}[h!]
\centering
 \begin{tabular}{c||c|c|c|c|c|}
  $H$& $\chi$ &  $\alpha$ & $\beta$ &$\gamma$ &$\xi$ \\\hline\hline
  $e$&  0&-1& 1&1&0 \\\hline 
 $f$&  0 &1& -1&0 &1\\\hline
  $g$  &  0 &0 &-1&1   &1\\  \hline
   $h$  &  1   &0 &1&-1 &0\\  \hline
 $j$ &   1   &1&0&-1&0 \\\hline 
 &   &&&&
 \end{tabular}
 \vskip1em
 \caption{}
 \label{n=5}
 \end{table}

As an example consider the hypergraph $H$ with vertex set  $V=\{\alpha,\beta,\gamma,\chi,\xi\}$ including
 the edges $e,f,g,h$ and $j$ marked as in Table~\ref{n=5}. Then
 $SF(H)$ is a triangle on the compound vertices $\{f,g\}, \{h,j\}$ and $\{e\}$.
 
 An agreeing linear order $L$ of $H$ defines a natural $2$-coloring of the edges  of $H$:
 color $A$ or $B$ will be assigned to $e\in E(H)$ if $\widehat{e}$ is $L$-minimal or $L$-maximal in $e$, respectively. 
 In other words, $E(H)=E_A\cup E_B$, where $E_A$ are edges of a min-marked subhypergraph of $H$ and  $E_B$ are edges of a `max-marked' subhypergraph of $H$, which is equivalent to a min-marked hypergraph.
 
{To obtain a characterization similar to the one in Proposition ~\ref{digraph} for
min-marked hypergraphs, assume that $SF(H)$ is two-colorable (bipartite), we consider any proper two-coloring with $A$ and $B$, and 
we associate to this $A,B$-coloring an auxiliary directed graph $AB(H)$ on vertex set $V(H)$ as follows.
 Each compound vertex represents a class of `$S$-equivalent edges' of  $H$;
assign the color of a compound vertex $\gamma$ of $SF(H)$ to all edges belonging to the class represented by $\gamma$.
Thus we obtain a partition $E(H)=E_A\cup E_B$, where 
$E_A=\{e\in E(H) :  \hbox{\; if  $e$ is colored with $A$}\}$ and 
$E_B=\{e\in E(H) :   \hbox{\; if  $e$ is colored with $B$}\}$. 
 For every $e\in E(H)$ and $\alpha,\beta\in e$, 
$\alpha\rightarrow\beta$ is an arc, if either $e\in E_A$ and $\widehat{e}=\alpha$, or    
$e\in E_B$ and $\widehat{e}=\beta$. }

On the analogy of Proposition ~\ref{digraph} we obtain a straightforward specification of $1$-extreme marked hypergraphs
in terms of auxiliary graphs.
\begin{proposition}
\label{1exchar}
A $1$-extreme marked hypergraph $H$ has an agreeing linear order if and only if
{$SF(H)$ contains no odd cycle and there is an $AB(H)$ graph associated with some proper two-coloring of $SF(H)$ that contains no  directed cycle.}
\end{proposition}

\begin{proof} Let $L$ be an agreeing linear order of $H$. As described earlier, $L$ defines a two-coloring, that is,  
a partition $E(H)=E_A\cup E_B$, where $E_A=\{e\in E(H) : \widehat{e} \hbox{\; is $L$-minimal in } e\}$ and 
$E_B=\{e\in E(H) : \widehat{e} \hbox{\; is $L$-maximal in } e\}$. 
{Notice that the color is the same for all edges represented by any given
compound vertex $\gamma$ of $SF(H)$, therefore  $SF(H)$ is a bipartite graph (and contains no odd cycle as required). Consider the $AB(H)$ graph associated with this two-coloring.} Let $\alpha, \beta\in e$, for some $e\in E(H)$.
If $\alpha\rightarrow\beta$ is an arc in $AB(H)$, then $\alpha <_L  \beta$, because $L$ agrees with $e$. Therefore, $AB(H)$ contains no directed cycle.

Next we prove that the requirements on the auxiliary graphs are sufficient for the existence of an agreeing linear order.
{The graph $SF(H)$ is bipartite, so its vertices have a proper two-coloring.   
Since $AB(H)$ associated with some $A,B$-coloring of $SF(H)$}, it has no directed cycle, there is a linear ordering $L$ on $V(H)$ such that $\alpha\rightarrow\beta$ implies
$\alpha <_L  \beta$. Observe that if $e\in E_A$ then for every $\beta\in e\setminus \{\widehat{e}\}$ we have $ \widehat{e}\rightarrow\beta$,
that is $ \widehat{e}$ is $L$-minimal in $e$.
Similarly, if $e\in E_B$ then for every $\alpha\in e\setminus\{ \widehat{e}\}$ we have $\alpha\rightarrow \widehat{e}$, hence
$ \widehat{e}$ is $L$-maximal in $e$.
\end{proof}

\subsection{$1$-extreme marked cliques}

{We will see here that the characterization of $1$-extreme marked hypergraphs in Proposition~\ref{1exchar} leads to the somewhat unexpected 
outcome that no Helly-type theorem exists for $1$-extreme marked cliques. }

\begin{proposition}
\label{general}
For every $r\geq 3$ and $n\geq r+1$ such that $n-r$ is even, there exists a   $1$-extreme marked $r$-uniform clique $H$ on $n$ vertices such that any subhypergraph of $H$ on $n-1$ vertices has an  agreeing linear order but $H$ does not.
\end{proposition}
\begin{proof}
We provide a construction for the case $r=3$, which will be extended for every $r\geq 4$.

\emph{Case $r=3$.} Notice that, by assumption, $n$ is odd.\\
Let $V(H)=\{1,2,\ldots,n\}$, and  let $e_i=\{i,i+1,i+2\}$  be edges with marked vertex $\widehat{e_i}=i+1$, for $i=1,2,\ldots,n$,  (modulo~$n$). For each $3$-set $f=\{\alpha_1,\alpha_2,\alpha_3\}\subseteq V$, with $\alpha_1<\alpha_2<\alpha_3$, different from all $e_i$, $i=1,2,\ldots,n$, let  $\widehat{f}=\alpha_1$.

Suppose an agreeing linear order $L$ exists. Let $C$ be the (odd) cycle on the vertex set $V(H)$ with edges $\{i,i+1\}$ (cyclically) for each $i$. The linear order $L$ induces a $2$-coloring on the edges as follows: $\{i,i+1\}$ is red, if $i>_L i+1$, and blue if $i<_L i+1$. The condition $\widehat{e_i}=i+1$ implies that this is a proper edge coloring of an odd cycle with $2$ colors, a contradiction.

Now let $\xi\in V(H)$, and let $H'=H-\xi$. Organize the vertices of $H'$ along a polygonal path, as shown below. (In this picture $\xi$ is even. If $\xi$ is odd, the only difference is the up/down position of the first and last point.)

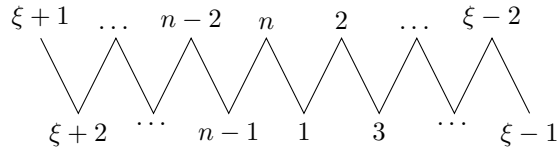
\begin{center}
\begin{tikzpicture}
\draw 
(-1,0) node [anchor=south] {$\xi+1$} 
-- (-.5,-1) node [anchor=north] {$\xi+2$} 
--(0,0) node [anchor=south] {\dots} 
-- (0.5,-1) node [anchor=north] {\dots} 
-- (1,0) node [anchor=south] {$n-2$} 
-- (1.5,-1) node [anchor=north] {$n-1$} 
-- (2,0) node [anchor=south] {$n$} 
--  (2.5,-1) node [anchor=north] {$1$} 
 -- (3,0) node [anchor=south] {$2$} 
 -- (3.5,-1) node [anchor=north] {$3$} 
 -- (4,0) node [anchor=south] {\dots} 
 -- (4.5,-1) node [anchor=north] {\dots}
 -- (5,0) node [anchor=south] {$\xi-2$} 
 -- (5.5,-1) node [anchor=north] {$\xi-1$};
\end{tikzpicture}
\end{center}

Then list the lower elements on the picture in increasing order, and list the upper elements in decreasing order, i.e., for $\xi$ even, let
$$L=(1,3,\ldots,\xi-1,\xi+2,\xi+4,\ldots,n-1,n,n-2,\ldots,\xi+1,\xi-2,\ldots,4,2).$$


 
The linear order $L$ clearly agrees with the edges $e_i\in E(H^\prime)$. For any other edge, if the marked vertex (i.e.\ lowest indexed vertex) is downstairs, it will be least in $L$; if it is upstairs, it will be greatest in $L$.

\emph{Case $r\geq 4$.} We define a $1$-extreme marked $r$-regular clique  satisfying the requirements as follows. Notice that  the condition $n-r\equiv 0 \pmod 2$ implies that $n_0=n-r+3$ is odd. Let  
$V_0=\{1,2\ldots,n_0\}$, and let $V=V_0\cup W$, where $W$ is an $(r-3)$-set disjoint from $V_0$.
Use the construction with vertex set $\{1,2\ldots,n_0\}$ as described  
above for the case $r=3$ to obtain a $1$-extreme marked $3$-regular clique $H_0$ on vertex set $V_0$ with no agreeing linear order. 

  For every $3$-set $f_0\in E(H_0)$ define the $r$-set $f=f_0\cup W$ with marked vertex  $\widehat{f}=\widehat{f_0}$. For any other $r$-set $e\subseteq V$ we have $|e\cap V_0|\geq 4$; define $\widehat{e}$ to be the smallest value in  $e\cap V_0$. Thus we obtain a  $1$-extreme marked $r$-uniform clique $H$. 
  
Edges that contain $3$ elements of $V_0$ prevent an agreeing linear order by a similar argument as above.

If $\xi\in W$ then $H-\xi$ becomes a $1$-extreme marked hypergraph such that the ordering of the values $1,2,\ldots,n_0$ followed by the elements of $W$ in arbitrary order is an agreeing linear order of $H-\xi$. 

For $1\leq \xi\leq n_0$ an agreeing linear order for $H-\xi$ is obtained by 
using an agreeing linear order $L_0$ of $H_0-\xi$, and  
by inserting the elements of $W$ 
 between $n_0-1$ and $n_0$ in any order. Equivalently, though less rigorously, $W$ is inserted between the ``downstairs'' and the ``upstairs'' vertices of $V_0$.
\end{proof}

\section{min\&max-marked cliques}
\label{2marked}
Let $H$ be an $r$-uniform clique such that for every $e\in E(H)$ two distinct vertices are marked as a {\it min-vertex} and a {\it max-vertex}, denoted 
$A(e)$ and $B(e)$, respectively. This hypergraph will be called a {\it min{\rm\&}max-marked clique}. 

A linear order $L$ of the vertex set of a 
min\&max-marked clique $H$ is called an {\em agreeing linear order}, provided  
$A(e)<_Lv<_LB(e)$, for every $e\in E(H)$ and $v\in e\setminus\{A(e),B(e)\}$. We prove Theorem~\ref{leftrightentry} in the following form.
\begin{theorem}
\label{2char}
For $r\geq 3$ an $r$-uniform min{\rm \&}max-marked clique $H$ with $n\geq 2r-2$ vertices has an
agreeing linear order 
if and only if 
there is an agreeing linear order on every $(2r-2)$-element subset of $V(H)$. Furthermore, the number $2r-2$ in the statement cannot be lowered.
\end{theorem} 
\begin{proof} For the second statement we adapt the construction in Theorem~\ref{main4} to 
present a min\&max-marked clique $H$ on  $n=2r-2$ vertices  such that
the vertex set of each subhypergraph of $H$ with $2r-3$ vertices has an agreeing linear order but $H$ does not.

Let 
 $V(H)=\{v_1,v_2,\ldots,v_{2r-2}\}$; for $e_0=\{v_1,\ldots,v_{r-1},v_r\}$ define 
 $A(e_0)=v_{1},B(e_0)=v_{r-1}$, and for every   $e=\{v_{i_1},v_{i_2},\dots,v_{i_r}\}$, $1\leq i_1<i_2<\ldots i_r\leq 2r-2$,  different from $e_0$ define  $A(e)=v_{i_1},B(e)=v_{i_r}$. 
 Let  $e_1=\{v_{r-1},v_r,\ldots,v_{2r-2}\}$. 

In an agreeing linear order $L$ of $V(H)$ 
we have $v_r<_L B(e_0)=v_{r-1}=A(e_1)<_L v_r$, a contradiction. 
Observe next that $L_1=(v_1,\ldots,v_{r-1},v_r,\ldots,v_{2r-2})$ agrees with all edges of $H-e_0$, and 
$L_0=(v_1,\ldots,v_{r},v_{r-1},\ldots,v_{2r-2})$  agrees with all edges of $H-e_1$.
 Because  $e_0\cup e_1=V(H)$, no subhypergraph $H-v_i$ contains both $e_0$ and $e_1$. Therefore, either $L_0-v_i$ or  $L_1-v_i$ is an agreeing linear order of $V(H)\setminus\{v_i\}$, for every $1\leq i\leq 2r-2$.
 
 \vskip.1cm
 To prove the first part of the theorem define a directed graph $G$ on $V=V(H)$ with an edge $(x,y)$ from $x$ to $y$ if 
 either $x=A(e)$ or $y=B(e)$, for some $e\in E(H)$, $x, y\in e$.
 Now $H$ has an agreeing linear order if and only if the vertices of $G$ have a labeling $v_1,v_2,\ldots, v_n$ such that each arc $(v_i,v_j)$ of $G$ implies $i<j$. This labeling of $V$ exists if and only if  $G$ has no directed cycle. Assume now that $H$ satisfies the Helly-condition. 
Observe that $G$ has  no directed $2$-cycle, because if it exists and is induced by $e,f\in E(H)$, then $|e\cup f|\leq 2r-2$, contradicting the condition that on 
every $(2r-2)$-element subset of $V(H)$ there is an agreeing linear order.

Assume to the contrary that $G$ contains a directed cycle, 
let $C=(a_1,a_2,\ldots,a_k)$ be a shortest directed cycle of $G$, $k\geq 3$. 

Let the arc $(a_1,a_{2})$ of $G$ be induced by  some $e\in E(H)$; without loss of generality we assume that 
$a_1=A(e)$. A shortest directed cycle has no chord, thus $e\cap C=\{a_1,a_{2}\}$. For the same reason, $C$ contains no $r$-tuple, hence $k\leq r-1$ and $|e\cup C|= (r+k)-2\geq r+2$. 

Let $f\subseteq e\cup C$ be an $r$-tuple containing $C$  and let $a=A(f)$.
If $a\in e\setminus C$, then  $(a,a_1)$ is a directed $2$-cycle;  if  $a=a_i$, $i=1,\ldots,k$, then 
 $(a,a_{i-1})$  is a directed $2$-cycle (with $a_0=a_k$), a contradiction.
\end{proof}
In addition to the direct proof of Theorem~\ref{leftrightentry} (see Theorem~\ref{2char} above)
we show how Theorem~\ref{leftrightentry}  follows as a corollary of Theorem~\ref{main}.  

\begin{proof}[Second proof of Theorem~\ref{2char}] 
Let $H$ be an $r$-uniform min\&max-marked clique with at least $2r-2$ vertices, and assume that there is an agreeing linear order on every $(2r-2)$-element subset of $V(H)$. Construct $H'$, a $2$-extreme marked clique from $H$ by making the marked vertices undistinguished, and apply Theorem~\ref{main} to get an agreeing linear order $L$, in the sense of Theorem~\ref{main}.

Every edge $e\in E(H)$ has the property that the smallest and the largest vertex of $e$ in $L$ form the set $\{A(e),B(e)\}$. Call $e$ ``good'', if the smallest vertex of $e$ in $L$ is $A(e)$, and the largest vertex is $B(e)$, and call it ``bad'' if it is the other way around. If every edge is good, then $L$ is an agreeing linear order in the sense of Theorem~\ref{2char}, so we are done. If every edge is bad, then the dual of $L$ will work, and we are done. So we just have to settle the case when there is a good edge $g$, and a bad edge $b$.

Recall that the Johnson graph  (see \cite{Alsp}) is defined on the $r$-element subsets of $\{1,\ldots,n\}$, as vertices, and two of these vertices ($r$-sets) $e$ and $f$ are adjacent, if $|e\cap f|=r-1$. An immediate observation is that the Johnson graph is connected.

On the path in the Johnson graph from the good edge $g$ to the bad edge $b$ (which are vertices of the Johnson graph), there are two elements $e,f\in E(H)$ such that $|e\cap f|=r-1$, and $e$ is good and $f$ is bad. Since $|e\cup f|=r+1\leq 2r-2$, there is an agreeing linear order $L'$ on $e\cup f$.

Note that $|\{A(e),B(e)\}\cap\{A(f),B(f)\}|\geq 1$. Without loss of generality, either $A(e)=B(f)$, or $A(e)=A(f)$. In the former case 
\[
A(f)<_{L'}B(f)=A(e)<_{L'}B(e),
\]
while in the latter,
\[
B(f)<_L A(f)=A(e)<_L B(e).
\]
In both cases, every element of $e\cap f$ should be both between $A(f)$ and $B(f)$, and $A(e)$ and $B(e)$, contradiction. 
\end{proof}

\section{Concluding remarks}

In this paper we established Helly-type results for the existence of an agreeing linear order,
for {\em complete} $r$-uniform hypergraphs, in four versions.
For 2-extreme marked, min-marked and min{\rm \&}max marked cliques we found Helly-type theorems and the Helly numbers.
For 1-extreme marked cliques we showed that there is no such theorem.

As a possible generalization of the questions discussed in the paper, 
D.~P\'alv\"olgyi (personal communication) proposed the question
of investigating Helly-type properties of cliques,
such that a poset is specified for each edge as ``marks''. An edge $e$ agrees with a linear order $L$, if $L|_e$ is a linear extension of the poset corresponding to $e$, and an agreeing linear order, as before, is a linear order on the vertex set that agrees with every edge.

Similarly to the proof of Theorem \ref{2char}, it can be shown that if none of the posets are antichains, and every set of $2r^2$ vertices has an agreeing order, then
there is an agreeing linear order of the vertices. 
However, determining the Helly-numbers (based on the posets) is wide open.

Thinking of $2$-extreme marked $3$-uniform hypergraphs, it is
a basic assumption (see \cite{Hunt,HuntK}) defining the betweenness relation
for all triples with prescribed `boundaries'. 
However, it is natural to ask, what is the situation for not necessarily complete hypergraphs. 
In this general case there is no Helly-type theorem in any of the four versions.
We show an example in the min-marked case.
For $m>2$ arbitrary, let  $u_1, \ldots, u_m$, $v^i_j$, $1\le i\le m$, $1\le j\le r-2$ be the vertices of $H$.
The hyperedges are $e_i=\{ u_i, u_{i+1}, v^i_1, \ldots, v^i_{r-2}\}$, $1\le i\le m$, cyclically.
Let $A(e_i)=u_i$. Clearly, there is no agreeing linear order for $H$, because the vertices 
$u_1, \ldots, u_m$ cannot be ordered in a desired way. However, if we remove any vertex, at least
one of the hyperedges is also deleted and then there is an agreeing order.

There are similarly easy counterexamples in the other cases too. 
But then, it is an interesting problem to find conditions on the original hypergraph that would still guarantee a Helly-type theorem.
In the 2-extreme marked, min-marked and min{\rm \&}max marked cases, is it enough to assume that
the original $r$-uniform hypergraph is dense (that is, the hypergraph has $\Omega(n^r)$ hyperedges)?



\end{document}